\documentclass[12pt,reqno]{amsart}

\numberwithin{equation}{section}

\newtheorem{teo}{Theorem }[section]

\newtheorem{lem}[teo]{Lemma}

\newcommand{\N}{{\mathbb N}}
\DeclareMathOperator*{\esssup}{ess\,sup}

\begin{document}

%%%%%%%%%%%%%%%%%%%%%%%%%%%

\title[ Local Dispersive Equations ]
      {Generalized dispersive equations of higher orders posed on bounded intervals: local theory }
      
\author{N. A. Larkin	\& 
J. Luchesi$^\dag$}

\address{Nikolai A. Larkin \newline
Departamento de Matem\'{a}tica, Universidade Estadual de Maring\'{a},
Av. Colombo 5790: Ag\^{e}ncia UEM, 87020-900, Maring\'{a}, PR, Brazil}
\email{nlarkine@uem.br}

\address{Jackson Luchesi \newline
Departamento de Matem\'{a}tica, Universidade Tecnol\'{o}gica Federal do Paran\'{a} - C\^{a}mpus Pato Branco, Via do Conhecimento Km 1, 85503-390, Pato Branco, PR, Brazil} \email{jacksonluchesi@utfpr.edu.br}

\keywords {Higher-order dispersive equations, local solutions, bounded domains.}
\thanks{}
\
\thanks{\it {2010 Mathematics Subject Classification: 35M20, 35Q72.} }
\thanks{$^\dag$ Corresponding author}

\begin{abstract}
	Initial-boundary value problems for nonlinear dispersive equations of evolution of order $2l+1, \;l\in\mathbb{N}$ with a convective term of the form $u^ku_x,\;k\in\mathbb{N}$ have been considered on intervals $(0,L),\;L\in (0,+\infty).$
	The existence and uniqueness of local regular solutions have been established.

\end{abstract}

\maketitle

\section{Introduction}

Our goal in this paper is to study solvability of initial-boundary value problems for  one-dimensional generalized dispersive equations of higher orders posed on a bounded interval
\begin{equation}\label{1.1}
u_t+\sum_{j=1}^{l}(-1)^{j+1}D_x^{2j+1}u+u^kD_xu=0 \,\,\, \mbox{in}\,\,\, Q_T,
\end{equation}
where $x\in (0,L),\; Q_T=(0,T)\times (0,L);\;\;l,k\in \mathbb{N};\;\;T,L$ are real positive numbers. This equation includes as special cases classical dispersive equations: when $l=k=1$, we have the well-known Korteweg-de Vries (KdV) equation, see \cite{kato,lar,saut},  and when $k=1,l=2$, we have the Kawahara equation \cite{doronin2,kawa, marcio}. For  $k=1$, the Cauchy problem for dispersive equations of higher orders has been studied in \cite{ biagioni,Fam1,huo, linponce,ponce1,pilod,tao} and initial boundary value problems have been studied in \cite{chile, doronin2,familark,marcio, Sangare}.  Although dispersive equations were deduced for the whole real line, necessity to calculate numerically the Cauchy problem, approximating the real line by finite intervals \cite{chile}, implies to study initial-boundary value problems posed on bounded and unbounded intervals, \cite{doronin2,familark,kuvsh,lar2,marcio}.
What concerns \eqref{1.1} with $k>1, \;l=1$, called generalized Korteweg-de Vries equations, the Cauchy problem  for \eqref{1.1} has been studied in \cite{farah,Fonseka1,Fonseka2,ponce2,martel,merle}, where was proved that for $k=4$, called the critical case,  the initial problem is well-posed for small  initial data, whereas for arbitrary initial data, solutions may blow-up in a finite time. The generalized KdV equation was intensively studied in order to understand the interaction between the dispersive term and  nonlinearity in the context of the theory of nonlinear dispersive evolution equations \cite{jeffrey,kaku,Rosier}. In \cite{lipaz}, an initial-boundary value problem for the generalized KdV equation with an internal damping posed on a bounded interval was studied in the critical case. Exponential decay of weak solutions for small initial data has been established. In \cite{araruna}, decay of weak solutions for $l=2,\;k=2$ was established.\\
Our goal in this work is to prove the existence and uniqueness of local regular solutions for all $k\in \mathbb{N}$ and for all finite positive $L$. \\
Our paper has the following structure: Chapter 1 is Introduction. Chapter 2 contains notations and auxiliary facts. In Chapter 3, formulation of  problems to be considered is given. In Chapter 4, we prove local existence and uniqueness of regular solutions. 

\section{Notations and auxiliary facts}

Let $x\in (0,L),\;\; D^i=\frac{\partial^i}{\partial x^i},\;i\in {\N};\;D=D^1.$  As in \cite{Adams} p. 23, we denote for scalar functions $f(x)$  the Banach space $L^p(0,L),\;1\leq p\leq+\infty \;$ with the norm: 
$$\| f \|_{L^p(0,L)}^p = \int_0^L |f(x)|^p\, dx,\; p\in[1,+\infty),\;
\|f\|_{\infty}=\esssup_{x\in(0, L)}|f(x)|.$$
For $p=2,\; L^2(0,L)$ is a Hilbert space  with the scalar product 
$$(u,v)=\int_0^L u(x)v(x) dx\; \mbox{and the norm}\; \|u\|^2=\int_0^L |u(x)|^2 dx.$$
The Sobolev space $W^{m,p}(0,L),\;m\in {\N}$ is a Banach space with the norm:
$$\| u \|_{W^{m,p}(0,L)}^p = \sum_{0 \leq |\alpha| \leq m} \|D^\alpha u \|_{L^p(0,L)}^p, \; 1\leq p<+\infty.$$
When $p=2,\;W^{m,2}(0,L)=H^m(0,L)$ is a Hilbert space with the following scalar product and the norm:
$$((u,v))_{H^m(0,L)}=\sum_{0\leq|j|\leq m}(D^j u,D^j v),\; \|u\|^2_{H^m(0,L)}=\sum_{0\leq|j|\leq m}\|D^j u\|^2.$$

Let $\mathcal{D}(0,L)$ or $\mathcal{D}([0,L])$ be the space of $C^{\infty}$ functions with compact support in $(0,L)$ or $[0,L]$. The closure of $C^{\infty}$ functions in $
W^{m,p}(0,L)$ is denoted by $W^{m,p}_0(0,L)$ and ($H^m_0(0,L)$ when $p=2$). For any space of functions, defined on an interval $(0,L)$, we omit the symbol $(0,L)$, for example, $L^p=L^p(0,L)$, $H^m=H^m(0,L)$, $H_0^m=H_0^m(0,L)$ etc.\\
We use the following version of the Gagliardo-Nirenberg inequality:
\begin{lem}\label{thm3.1}
	Let $u$ belong to $H_0^l(0,L)$, then the following inequality holds:
	\begin{equation}\label{2.4}
	\|u\|_{\infty}\leq \sqrt{2}\|D^lu\|^{\frac{1}{2l}}\|u\|^{1-\frac{1}{2l}}.
	\end{equation}
\end{lem}
\begin{proof}
	Let $l=1$, then for any $x\in (0,L)$
	\begin{align*}
		u^2(x) & =  \int_{0}^{x}D[u^2(\xi)]d\xi\leq 2\int_{0}^{x}|u(\xi)||D(\xi)|d\xi\\
		& \leq  2\int_0^L |u||Du|dx\leq 2\|u\|\|Du\|
	\end{align*}
and $\|u\|_{\infty}\leq \sqrt{2}\|Du\|^{\frac{1}{2}}\|u\|^{\frac{1}{2}}$.\\
	For $l\geq 2$ the proof will be done in several steps:\\	
	{\bf\underline{Step 1}:} $\|Du\|\leq \|D^2u\|^{\frac{1}{2}}\|u\|^{\frac{1}{2}}$, for all $u\in H_0^2(0,L)$.\\
	Let $u\in H_0^2(0,L)$, then 
	$$\|Du\|^2=\int_{0}^{L}(Du)^2dx=-\int_{0}^{L}uD^2udx\leq \int_{0}^{L}|u||D^2u|dx\leq \|u\|\|D^2u\|,$$
	and Step 1 is proved.\\
	{\bf\underline{Step 2}:} Let $u\in H_0^l(0,L)$, then $\|Du\|\leq \|D^lu\|^{\frac{1}{l}}\|u\|^{1-\frac{1}{l}}$, for all $l\geq 2$.\\
	We proceed by induction on $l$. The case $l=2$ is proved in Step 1. Suppose the result is valid for $m>2$, then if $u\in H_0^{m+1}(0,L)$, we have $u\in H_0^2(0,L)$ and $Du\in H_0^m(0,L)$. Consequently,
	\begin{equation}\label{2.5}
	\|Du\|\leq \|D^2u\|^{\frac{1}{2}}\|u\|^{\frac{1}{2}}
	\end{equation}
	and
	\begin{equation}\label{2.6}
	\|D^2u\|\leq \|D^{m+1}u\|^{\frac{1}{m}}\|Du\|^{1-\frac{1}{m}}.
	\end{equation}
	Substituting \eqref{2.6} into \eqref{2.5}, we obtain
	\begin{equation*}
	\|Du\|\leq \|D^{m+1}u\|^{\frac{1}{m+1}}\|Du\|^{1-\frac{1}{m+1}}.
	\end{equation*}
	Therefore $\|Du\|\leq \|D^lu\|^{\frac{1}{l}}\|u\|^{1-\frac{1}{l}}$  for all $l\geq 2$.\\
	{\bf\underline{Step 3}:} $\|u\|_{\infty}\leq\sqrt{2} \|D^lu\|^{\frac{1}{2l}}\|u\|^{1-\frac{1}{2l}}$, for all $u\in H_0^l(0,L)$ and $l\geq 2$.\\
	Let $l\geq 2$ and $u\in H_0^l(0,L)$. Then by the case  $l=1$ and Step 2,
	\begin{align*}
	\|u\|_{\infty} 
	\leq  \sqrt{2} \|D^lu\|^{\frac{1}{2l}}\|u\|^{\left(1-\frac{1}{l}\right)\frac{1}{2}}\|u\|^{\frac{1}{2}}
	=  \sqrt{2} \|D^lu\|^{\frac{1}{2l}}\|u\|^{1-\frac{1}{2l}}.
	\end{align*}
The proof of Lemma 2.1 is complete.
\end{proof}	
\begin{lem}\label{*}(See \cite{niren}, p. 125).
	Suppose $u$ and $D^mu$, $m\in\mathbb{N}$ belong to $L^2(0,L)$. Then for the derivatives $D^iu$, $0\leq i<m$, the following inequality holds: 
	\begin{equation}\label{2.7}
	\|D^iu\|\leq A_1\|D^mu\|^{\frac{i}{m}}\|u\|^{1-\frac{i}{m}}+A_2\|u\|,
	\end{equation}
where $A_1$, $A_2$ are constants depending only on $L$, $m$, $i$.
	\end{lem}
\section{Formulation of the problem }
Let $T$ and $L$ be real positive numbers and $Q_T=\{(t,x)\in\mathbb{R}^2: t\in(0,T),\,\, x\in(0,L)\}$. Consider the following higher-order dispersive equation:
\begin{equation}\label{3.1}
u_t+\sum_{j=1}^{l}(-1)^{j+1}D^{2j+1}u+u^kDu=0 \,\,\, \mbox{in}\,\,\, Q_T
\end{equation}
subject to initial-boundary conditions:
\begin{equation} \label{3.2}
u(0,x)=u_0(x), \,\, x\in(0,L); \\
\end{equation}
\begin{equation} \label{3.3}
D^{i}u(t,0)=D^{i}u(t,L)=D^{l}u(t,L)=0,\,\, i=0,\ldots, l-1, \\
\end{equation}
where $l,k\in \mathbb{N}$ and $u_0$ is a given function.

\section{Local solutions}
We start with the linearized version of \eqref{3.1}
\begin{equation}\label{4.1}
u_t+\sum_{j=1}^{l}(-1)^{j+1}D^{2j+1}u=f\,\,\, \mbox{in}\,\,\, Q_T
\end{equation}
subject to initial-boundary conditions \eqref{3.2}-\eqref{3.3}. Define the linear differential operator in $L^2(0,L)$:
$$Au\equiv\sum_{j=1}^{l}(-1)^{j+1}D^{2j+1}u,$$
$$D(A)\equiv \{u\in H_0^l(0,L)\cap H^{2l+1}(0,L): D^{l}u(L)=0\}.$$
\begin{teo}\label{thm4.1}(See \cite{larluch}, Theorem 4.1.) Let $g\in L^2(0,L)$. Then for all $a>0$ the stationary equation: $au+Au=g$ in $(0,L)$, subject to boundary conditions: $D^{i}u(0)=D^{i}u(L)=D^{l}u(L)=0,\,\, i=0,\ldots, l-1$ admits a unique regular solution $u\in H^{2l+1}(0,L)$ such that
	\begin{equation}\label{4.2}
		\|u\|_{H^{2l+1}}\leq C \|g\|
	\end{equation}
	with the constant $C$ depending only on $L$, $l$ and $a$.
\end{teo}
\begin{teo}\label{thm4.2} Let $u_0\in D(A)$ and $f\in H^1(0,T;L^2(0,L))$. Then for every $T>0$, problem \eqref{4.1},\eqref{3.2},\eqref{3.3} has a unique solution $u=u(t,x)$ such that
	$$u\in C([0,T];D(A))\cap C^1([0,T];L^2(0,L)).$$  
\end{teo}
\begin{proof}
According to Theorem 4.1, the operator $aI+A$ is surjective for all $a>0$. Moreover, if $u\in D(A)$ then
$$(Au,u)=\frac{1}{2}(D^lu(0))^2\geq 0$$ 
and the result follows by the semigroup theory. (See \cite{zheng}, Lemma 2.2.3 and Corollary 2.4.2.)
\end{proof}	
\begin{teo}\label{thm4.3} Let $u_0\in D(A)$. Then there exists a real $T_*\in(0,T]$ such that  \eqref{3.1}-\eqref{3.3} has a unique regular solution $u=u(t,x)$:
	$$u\in L^{\infty}(0,T_*; H_0^l(0,L)\cap H^{2l+1}(0,L))\cap L^2(0,T_*; H^{(2l+1)+l}(0,L)),$$
	$$u_t\in L^{\infty}(0,T_*; L^2(0,L))\cap L^2(0,T_*; H_0^l(0,L)).$$
\end{teo}
\begin{proof}
The proof will be done using the Banach fixed point theorem.
First, let $v,v_t\in X= L^{\infty}(0,T;L^2(0,L))\cap L^2(0,T;H_0^l(0,L))$, then  $v^kDv\in H^1(0,T; L^2(0,L))$. Indeed, since $v,v_t\in L^2(0,T;H_0^l(0,L))$, we have $v\in C([0,T];H_0^l(0,L))$ and by \eqref{2.4}:
\begin{align*}
\|v^kDv\|_{L^2(Q_T)}^2 &\leq  \int_{0}^{T}\|v\|_{\infty}^{2k}(t)\|Dv\|^2(t)dt
  \leq \int_{0}^{T}2^k\|v\|_{H_0^l}^{2k}(t)\|v\|_{H_0^l}^2(t)dt\\
  &\leq 2^k\|v\|_{C([0,T];H_0^l(0,L))}^{2k}\|v\|_{L^2(0,T;H_0^l(0,L))}^2.	
\end{align*}
On the other hand, $(v^kDv)_t=kv^{k-1}v_tDv+v^kDv_t$. Hence
$$\|(v^kDv)_t\|_{L^2(Q_T)}^2\leq \underbrace{2k^2\|v^{k-1}v_tDv\|_{L^2(Q_T)}^2}_{I_1}+\underbrace{2\|v^kDv_t\|_{L^2(Q_T)}^2}_{I_2}.$$
We estimate
\begin{align*}
I_1 & \leq  2k^2\int_{0}^{T}\left[\|v\|_{\infty}^{2(k-1)}(t)\|v_t\|_{\infty}^2(t)\|Dv\|^2(t)\right]dt\\
 & \leq  2k^2 \int_{0}^{T}\left[2^{k-1}\|v\|_{H_0^l}^{2(k-1)}(t)2\|v_t\|_{H_0^l}^2(t)\|v\|_{H_0^l}^2(t)\right]dt\\
 & \leq  2^{k+1}k^2\|v\|_{C([0,T];H_0^l(0,L))}^{2k}\|v_t\|_{L^2(0,T;H_0^l(0,L))}^2,
\end{align*}
\begin{align*}
	I_2  \leq 2\int_{0}^{T}\|v\|_{\infty}^{2k}(t)\|Dv_t\|^2(t)dt
	 \leq  2^{k+1}\|v\|_{C([0,T];H_0^l(0,L))}^{2k}\|v_t\|_{L^2(0,T;H_0^l(0,L))}^2.
\end{align*}
Taking $f=-v^kDv$ in Theorem 4.2, define an operator $P$, related to \eqref{4.1},\eqref{3.2},\eqref{3.3} such that $v\mapsto u=Pv$, and the space:
$$V=\{v(t,x):v,v_t\in X;\,\, v(0,\cdot)\equiv u_0\}$$
with the norm
$$\|v\|_V^2=\esssup_{t\in(0, T)}\{\|v\|^2(t)+\|v_t\|^2(t)\}+\int_{0}^{T}\sum_{j=1}^{l}[\|D^iv\|^2(t)+\|D^iv_t\|^2(t)]dt.$$
Consider in $V$ a ball
$$B_R=\{v\in V: \|v\|_V^2\leq 8R^2\},$$
where $R>0$ satisfies the inequality:
\begin{equation}\label{4.3}
(1+L)(1+l)\left(2^k\|u_0\|_{H_0^l}^{2(k+1)}+\|u_0\|_{H^{2l+1}}^2\right)\leq R^2.
\end{equation}
\begin{lem}\label{lem4.4} There is a real $0<T_0\leq T$ such that the operator $P$ maps $B_R$ into itself.	
\end{lem}
\begin{proof} First, if $v\in B_R$, then due to \eqref{4.3}, we have
\begin{equation}\label{4.4}
	\|Dv\|^2(t) \leq \|Du_0\|^2+\int_{0}^{T}[\|Dv\|^2(t)+\|Dv_t\|^2(t)]dt 
	 \leq  9R^2.	
\end{equation}	
We will need the following estimates:
\subsection*{Estimate 1} Multiplying \eqref{4.1} by $2(1+x)u$ and integrating over $(0,L)$, we obtain	
\begin{align}\label{4.5}
\frac{d}{dt}(1+x,u^2)(t)&+\sum_{j=1}^{l}(2j+1)\|D^ju\|^2(t)+(D^lu(t,0))^2\notag \\
&= -2(v^kDv,(1+x)u)(t).
\end{align}
Making use of \eqref{2.4} and \eqref{4.4}, we estimate
\begin{align*}
-2(v^kDv,(1+x)u)(t)  &\leq  (1+L)\|v^kDv\|^2(t)+(1+x,u^2)(t)\\
 &\leq (1+L)\|v\|_{\infty}^{2k}(t)\|Dv\|^2(t)+(1+x,u^2)(t)\\
 &\leq  (1+L)2^k\|v\|_{H_0^1}^{2k}(t)\|Dv\|^2(t)+(1+x,u^2)(t)\\
 & \leq (1+L) 9(34)^kR^{2k+2}+(1+x,u^2)(t).
\end{align*}
Then \eqref{4.5} becomes
\begin{align}\label{4.6}
	\frac{d}{dt}(1+x,u^2)(t)&+\sum_{j=1}^{l}(2j+1)\|D^ju\|^2(t)+(D^lu(t,0))^2\notag \\
& \leq (1+x,u^2)(t)+ (1+L) 9(34)^kR^{2k+2}.
\end{align}
By the Gronwall Lemma and \eqref{4.3}, 
\begin{align*}
(1+x,u^2)(t) &\leq  e^T\left((1+x,u_0^2)+(1+L)9(34)^kR^{2k+2}T\right)\\
  &\leq  e^T\left(\frac{R^2}{2}+(1+L)9(34)^kR^{2k+2}T\right).
\end{align*}
Choosing $0<T_1\leq T$ such that $e^{T_1}\leq 2$ and $(1+L)9(34)^kR^{2k}T_1\leq \frac{1}{2}$, we conclude
\begin{equation}\label{4.7}
(1+x,u^2)(t)\leq 2\left(\frac{R^2}{2}+\frac{R^2}{2}\right)=2R^2, \,\,\, t\in(0,T_1].
\end{equation}
Substituting \eqref{4.7} into \eqref{4.6} and integrating over $(0,T_1)$, we get 
\begin{equation*}
3\int_{0}^{T_1}\sum_{j=1}^{l}\|D^ju\|^2(t)dt\leq R^2\left((1+L)9(34)^kR^{2k}+2\right)T_1+(1+x,u_0^2).
\end{equation*}
Taking $0<T_2\leq T_1$ such that $\left((1+L)9(34)^kR^{2k}+2\right)T_2\leq 2$ and making use of \eqref{4.3}, we conclude 
\begin{equation}\label{4.8}
\int_{0}^{T_2}\sum_{j=1}^{l}\|D^ju\|^2(t)dt\leq\frac{2R^2+R^2}{3}=R^2.
\end{equation}
\subsection*{Estimate 2} Differentiating \eqref{4.1} with respect to $t$, multiplying the result by $2(1+x)u_t$ and integrating over $(0,L)$, one gets
\begin{align}\label{4.9}	\frac{d}{dt}(1+x,u_t^2)(t)+\sum_{j=1}^{l}(2j+1)\|D^ju_t\|^2(t)+(D^lu_t(t,0))^2\notag\\= \underbrace{2(-kv^{k-1}v_tDv,(1+x)u_t)(t)}_{I_1}+\underbrace{2(-v^kDv_t,(1+x)u_t)(t)}_{I_2}.
\end{align}
Making use of \eqref{2.4} and \eqref{4.4}, we estimate for an arbitrary $\epsilon>0$:
\begin{align*}
I_1 & \leq  \epsilon (1+L)k\|v^{k-1}v_tDv\|^2(t)+\frac{1}{\epsilon}(1+x,u_t^2)(t)\\
 & \leq \epsilon (1+L) k \|v\|_{\infty}^{2(k-1)}(t)\|v_t\|_{\infty}^2(t)\|Dv\|^2(t)+\frac{1}{\epsilon}(1+x,u_t^2)(t)\\
  & \leq  \epsilon (1+L) k 2^{k}\|v\|_{H_0^1}^{2(k-1)}(t)\|v_t\|_{H_0^1}^2(t)\|Dv\|^2(t)+\frac{1}{\epsilon}(1+x,u_t^2)(t)\\
  & \leq  \epsilon (1+L) k 2^k(17R^2)^{k-1}\left(8R^2+\|Dv_t\|^2(t)\right)9R^2+\frac{1}{\epsilon}(1+x,u_t^2)(t),\\
 \end{align*}
 \begin{align*}
	I_2 & \leq  \epsilon (1+L) \|v^kDv_t\|^2(t)+\frac{1}{\epsilon}(1+x,u_t^2)(t)
	 \leq  \epsilon (1+L)\|v\|_{\infty}^{2k}(t)\|Dv_t\|^2(t)\\&+\frac{1}{\epsilon}(1+x,u_t^2)(t)
	 \leq \epsilon (1+L)2^k\|v\|_{H_0^1}^{2k}(t)\|Dv_t\|^2(t)+\frac{1}{\epsilon}(1+x,u_t^2)(t)\\
	 &\leq  \epsilon (1+L)(34)^kR^{2k}\|Dv_t\|^2(t)+\frac{1}{\epsilon}(1+x,u_t^2)(t).
\end{align*}
Substituting $I_1$, $I_2$ into \eqref{4.9} and taking $\epsilon=\left(16(1+L)\alpha_kR^{2k}\right)^{-1}$ with $\alpha_k=9k2^k(17)^{k-1}+(34)^k$, we reduce it to the inequality
\begin{align}\label{4.10}
	&\frac{d}{dt}(1+x,u_t^2)(t)+\sum_{j=1}^{l}(2j+1)\|D^ju_t\|^2(t)+(D^lu_t(t,0))^2 \leq \frac{1}{16}\|Dv_t\|^2(t) \notag \\ & +9k(34)^{k-1}R^2\alpha_k^{-1} +(32)(1+L)\alpha_kR^{2k}(1+x,u_t^2)(t).
		\end{align}
By the Gronwall Lemma,
\begin{align*}
&(1+x,u_t^2)(t)  \leq  e^{(32)(1+L)\alpha_kR^{2k}T}(1+x,u_t^2)(0)\notag\\  &+  e^{(32)(1+L)\alpha_kR^{2k}T}\left(9k(34)^{k-1}R^2\alpha_k^{-1}T+\frac{1}{16}\int_{0}^{T}\|Dv_t\|^2(t)dt\right).
\end{align*}
Due to \eqref{2.4},
\begin{align*}
\|u_0^kDu_0\|^2 & \leq  \|u\|_{\infty}^{2k}\|Du_0\|^2 \leq 2^k\|u_0\|^{\left(1-\frac{1}{2l}\right)2k}\|D^lu_0\|^{\frac{k}{l}}\|Du_0\|^2\\
 & \leq  2^k\|u_0\|_{H_0^l}^{\left(1-\frac{1}{2l}\right)2k}\|u_0\|_{H_0^l}^{\frac{k}{l}}\|u_0\|_{H_0^l}^2=2^k\|u_0\|_{H_0^l}^{2(k+1)},
\end{align*}
whence \eqref{4.3} implies
\begin{align*}
	&(1+x,u_t^2)(0)  \leq  (1+L)(1+l)\left(\|u_0^kDu_0\|^2+\|u_0\|_{H^{2l+1}}^2\right) \\
	 &\leq  (1+L)(1+l)\left(2^k\|u_0\|_{H_0^l}^{2(k+1)}+\|u_0\|_{H^{2l+1}}^2\right) 
	 \leq  R^2.
\end{align*}
Taking $0<T_3\leq T$ such that  $e^{(32)(1+L)\alpha_kR^{2k}T_3}\leq 2$,  
 $9k(34)^{k-1}\alpha_k^{-1}T_3\leq \frac{1}{2}$, we get
\begin{equation}\label{4.11}
(1+x,u_t^2)(t)\leq 2\left(R^2+\frac{R^2}{2}+\frac{R^2}{2}\right)=4R^2, \,\,\, t\in(0,T_3].
\end{equation}	
Substituting \eqref{4.11} into \eqref{4.10} and integrating over $(0,T_3)$, we obtain
\begin{align*}
	3\int_{0}^{T_3}\sum_{j=1}^{l}\|D^ju_t\|^2(t)dt& \leq  R^2\left((128)(1+L)\alpha_kR^{2k}+9k(34)^{k-1}\alpha_k^{-1}\right)T_3\\
	 &+  \frac{1}{16}\int_{0}^{T_3}\|Dv_t\|^2(t)dt+(1+x,u_t^2)(0).
\end{align*}
Let $0<T_4\leq T_3$ satisfy $\left((128)(1+L)\alpha_kR^{2k}+9k(34)^{k-1}\alpha_k^{-1}\right)T_4\leq 2$, then
\begin{equation}\label{4.12}
\int_{0}^{T_4}\sum_{j=1}^{l}\|D^ju_t\|^2(t)dt\leq \frac{1}{3}\left(2R^2+\frac{R^2}{2}+\frac{R^2}{2}\right)=R^2.
\end{equation}
Choosing $T_0=\min\{T_2,T_4\}$ and taking into account \eqref{4.7}, \eqref{4.8}, \eqref{4.11}, \eqref{4.12}, we complete the proof of Lemma 4.4.
\end{proof}	
\begin{lem}\label{lem4.5} There is a real $0<T_*\leq T_0$ such that the mapping $P$ is a contraction in $B_R$.\end{lem}
\begin{proof}
	For $v_1,v_2\in B_R$, denote $u_i=Pv_i, \, i=1,2$, $w=v_1-v_2$ and $z=u_1-u_2$.
	Then $z$ satisfies the equation
\begin{equation}\label{4.13}
z_t+\sum_{j=1}^{l}(-1)^{j+1}D^{2j+1}z=-v_1^kDw-(v_1^k-v_2^k)Dv_2 \,\,\, \mbox{in}\,\,\, Q_{T_0}
\end{equation}
and homogeneous boundary conditions \eqref{3.2} and initial data $z(0,\cdot)\equiv 0$.	
\subsection*{Estimate 3} Multiplying \eqref{4.13} by $2(1+x)z$ and integrating over $(0,L)$, we obtain
\begin{align}\label{4.14}
	\frac{d}{dt}(1+x,z^2)(t)+\sum_{j=1}^{l}(2j+1)\|D^jz\|^2(t)+(D^lz(t,0))^2\notag\\
=\underbrace{2(-v_1^kDw,(1+x)z)(t)}_{I_1}+\underbrace{2(-(v_1^k-v_2^k)Dv_2,(1+x)z)(t)}_{I_2}.
\end{align}
Making use of \eqref{2.4} and \eqref{4.4}, we estimate for an arbitrary $\epsilon>0$:
\begin{align*}
	I_1 & \leq  \epsilon (1+L)\|v_1^kDw\|^2(t)+\frac{1}{\epsilon}(1+x,z^2)(t)\\
	 &\leq  \epsilon(1+L) (34)^kR^{2k}\|Dw\|^2(t)+\frac{1}{\epsilon}(1+x,z^2)(t).
\end{align*}
In order to estimate $I_2$, we need the following inequality: (See details in \cite{larluch}.)
\begin{equation}\label{tvm}
|v_1^k-v_2^k|\leq k2^{(k-1)}(|v_1|^{(k-1)}+|v_2|^{(k-1)})|w|.
\end{equation}
Then
\begin{align*}
	I_2  &\leq  \epsilon (1+L)\|(v_1^k-v_2^k)Dv_2\|^2(t)+\frac{1}{\epsilon}(1+x,z^2)(t)\\
	 &\leq \epsilon (1+L)k^22^{2k-1}\sum_{i=1}^{2} \|v_i\|_{\infty}^{2(k-1)}(t)(|w|^2,|Dv_2|^2)(t)+\frac{1}{\epsilon}(1+x,z^2)(t)\\&\leq \epsilon(1+L) k^2  2^{3k-1}(17)^{k-1}R^{2k-2}\|w\|_{\infty}^2(t)\|Dv_2\|^2(t)+\frac{1}{\epsilon}(1+x,z^2)(t)	\\ 
	 &\leq \epsilon(1+L) 9k^2  2^{3k}(17)^{k-1}R^{2k}\|w\|_{H_0^1}^2(t)+\frac{1}{\epsilon}(1+x,z^2)(t).
\end{align*}
Taking $\epsilon=(16(1+L)\beta_kR^{2k})^{-1}$ with $\beta_k=9k^22^{3k}(17)^{k-1}+(34)^k$, we write \eqref{4.14} as 
\begin{align}\label{4.16}
		\frac{d}{dt}(1+x,z^2)(t)+\sum_{j=1}^{l}(2j+1)\|D^jz\|^2(t)+(D^lz(t,0))^2\notag\\
		\leq (32)(1+L)\beta_kR^{2k}(1+x,z^2)(t)+\frac{1}{16}\|w\|_{H_0^1}^2(t).
\end{align}
By the Gronwall Lemma,
\begin{align*}
(1+x,z^2)(t)  \leq  e^{(32)(1+L)\beta_kR^{2k}T_0}\left((1+x,z^2)(0)+\frac{1}{16}\int_{0}^{T_0}\|w\|_{H_0^1}^2(t)dt\right).\\
	\end{align*}
Choosing $0<T_5\leq T_0$ such that $e^{(32)(1+L)\beta_kR^{2k}T_5}\leq 2$, we conclude 
\begin{equation}\label{4.17}
(1+x,z^2)(t)\leq \frac{1}{8}\|w\|_V^2, \,\,\, t\in(0,T_5].
\end{equation}
Substituting \eqref{4.17} into \eqref{4.16} and integrating over $(0,T_5)$, we find
\begin{equation*}
3\int_{0}^{T_5}\sum_{j=1}^{l}\|D^jz\|^2(t)dt\leq 4(1+L)\beta_kR^{2k}\|w\|_V^2T_5+\frac{1}{16}\|w\|_V^2.
\end{equation*}
Taking $0<T_6\leq T_5$ such that $4(1+L)\beta_kR^{2k}T_6\leq \frac{5}{16}$, we get
\begin{equation}\label{4.18}
	\int_{0}^{T_6}\sum_{j=1}^{l}\|D^jz\|^2(t)dt\leq \frac{1}{3}\left(\frac{5}{16}+\frac{1}{16}\right)\|w\|_V^2=\frac{1}{8}\|w\|_V^2.
\end{equation}

\subsection*{Estimate 4} 
Differentiating \eqref{4.13} with respect to $t$, multiplying the result by $2(1+x)z_t$ and integrating over $(0,L)$, we obtain
\begin{align}\label{4.19}
&	\frac{d}{dt}(1+x,z_t^2)(t)+\sum_{j=1}^{l}(2j+1)\|D^jz_t\|^2(t)+(D^lz_t(t,0))^2\notag\\&\leq\underbrace{2(-kv_1^{k-1}v_{1t}Dw,(1+x)z_t)(t)}_{I_1}+\underbrace{2(-v_1^kDw_t,(1+x)z_t)(t)}_{I_2}+\notag\\
& \underbrace{2(-kv_1^{k-1}w_tDv_2,(1+x)z_t)(t)}_{I_3}+\underbrace{2(-k(v_1^{k-1}-v_2^{k-1})v_{2t}Dv_2,(1+x)z_t)(t)}_{I_4}\notag\\&+\underbrace{2(-(v_1^k-v_2^k)Dv_{2t},(1+x)z_t)(t)}_{I_5}.
\end{align}
Making use of \eqref{2.4},\eqref{4.4},\eqref{tvm}, we estimate for an arbitrary $\epsilon>0$:
\begin{align*}
	I_1& 
	 \leq  \epsilon(1+L) k 2^{k+3}(17)^{k-1}R^{2k}\|Dw\|^2(t) +\frac{1}{\epsilon}(1+x,z_t^2)(t)\\&+  \epsilon (1+L) k2^k(17)^{k-1}R^{2k-2}\|Dv_{1t}\|^2(t)\|Dw\|^2(t),
\end{align*}
\begin{align*}
	&I_2  \leq  \epsilon(1+L) (34)^kR^{2k}\|Dw_t\|^2(t)+\frac{1}{\epsilon}(1+x,z_t^2)(t),
\end{align*}
\begin{align*}
	I_3 \leq  \epsilon(1+L) 9k 2^k(17)^{k-1}R^{2k}\|w_t\|_{H_0^1}^2(t)+\frac{1}{\epsilon}(1+x,z_t^2)(t),
\end{align*}
\begin{align*}
	I_4  	 &\leq \epsilon (1+L)9k(k-1)^22^{3k}(17)^{k-2}R^{2k}\|w\|_{H_0^1}^2(t)+\frac{1}{\epsilon}(1+x,z_t^2)(t)\\
	&+ \epsilon(1+L) 9k(k-1)^2 2^{3(k-1)}(17)^{k-2}R^{2k-2}\|Dv_{2t}\|^2(t)\|w\|_{H_0^1}^2(t),
\end{align*}
\begin{align*}
	I_5   \leq  \epsilon(1+L) k^2 2^{3k-1}(17)^{k-1}R^{2k-2}\|Dv_{2t}\|^2(t)\|w\|_{H_0^1}^2(t)+\frac{1}{\epsilon}(1+x,z_t^2)(t).
\end{align*}
Taking $\epsilon=(16(1+L)\gamma_k R^{2k})^{-1}$, where $$\gamma_k=k2^k(17)^{k-1}\left(25+18(k-1)^24^k(17)^{-1}+k^{-1}(17)+k2^{2k+2}\right),$$ \eqref{4.19} becomes
\begin{align}\label{4.20}
	&	\frac{d}{dt}(1+x,z_t^2)(t)+\sum_{j=1}^{l}(2j+1)\|D^jz_t\|^2(t)+(D^lz_t(t,0))^2\notag\\&\leq(16\gamma_k)^{-1}\left(k 2^{k+3}(17)^{k-1}+k2^k(17)^{k-1}R^{-2}\|Dv_{1t}\|^2(t)\right)\|Dw\|^2(t)\notag\\&+(16\gamma_k)^{-1}\left((34)^k+9k 2^k(17)^{k-1}\right)\|w_t\|_{H_0^1}^2(t)\notag\\&+(16\gamma_k)^{-1}9k(k-1)^22^{3k}(17)^{k-2}\|w\|_{H_0^1}^2(t)\notag\\&+(16\gamma_k)^{-1}9k(k-1)^2 2^{3(k-1)}(17)^{k-2}R^{-2}\|Dv_{2t}\|^2(t)\|w\|_{H_0^1}^2(t)\notag\\&+(16\gamma_k)^{-1}k^2 2^{3k-1}(17)^{k-1}R^{-2}\|Dv_{2t}\|^2(t)\|w\|_{H_0^1}^2(t)\notag\\&+5(16)(1+L)\gamma_kR^{2k}(1+x,z_t^2)(t).
\end{align}
By the Gronwall Lemma, 
$$(1+x,z_t^2)(t)\leq e^{5(16)(1+L)\gamma_kR^{2k}T_0}\left((16\gamma_k)^{-1}\gamma_k\right)\|w\|_V^2.$$
Choosing $0<T_7\leq T_0$ such that $e^{5(16)(1+L)\gamma_kR^{2k}T_7}\leq 2$, we conclude 
\begin{equation}\label{4.21}
(1+x,z_t^2)(t)\leq \frac{1}{8}\|w\|_V^2, \,\,\, t\in(0,T_7].
\end{equation}
Substituting \eqref{4.21} into \eqref{4.20} and integrating over $(0,T_7)$, we find
$$3\int_{0}^{T_7}\sum_{j=1}^{l}\|D^jz_t\|^2(t)dt\leq (10)(1+L)\gamma_kR^{2k}\|w\|_V^2T_7+\frac{1}{16}\|w\|_V^2.$$
Choosing  $0<T_8\leq T_7$ such that $(10)(1+L)\gamma_kR^{2k}T_8\leq \frac{5}{16}$, we obtain
\begin{equation}\label{4.22}
\int_{0}^{T_8}\sum_{j=1}^{l}\|D^jz_t\|^2(t)dt\leq\frac{1}{3}\left(\frac{5}{16}+\frac{1}{16}\right)\|w\|_V^2=\frac{1}{8}\|w\|_V^2.
\end{equation}
Taking $T_*=\min\{T_6,T_8\}$ and making use of \eqref{4.17},\eqref{4.18},\eqref{4.21},\eqref{4.22}, we find that $\|z\|_V^2\leq \frac{1}{2}\|w\|_V^2$. Hence $P$ is a contraction in $B_R$ and Lemma 4.5 is thereby proved.
\end{proof}	
According to Lemmas	4.3 and 4.4 and the contraction principle, problem \eqref{3.1}-\eqref{3.3} has a unique  generalized solution $u=u(t,x)$:
\begin{equation}\label{4.23}
u,u_t\in L^{\infty}(0,T_*; L^2(0,L))\cap L^2(0,T_*; H_0^l(0,L)).
\end{equation}
Write \eqref{3.1} in the form
\begin{equation}\label{4.24}
u+\sum_{j=1}^{l}(-1)^{j+1}D^{2j+1}u=u-u_t-u^kDu=F(t,x).
\end{equation}
Since $u,u_t$ satisfy \eqref{4.23}, it follows that $u^kDu\in H^1(0,T_*; L^2(0,L))$, for all $k\in \mathbb{N}$, hence $F\in L^{\infty}(0,T_*;L^2(0,L))$. Due to \eqref{4.2}, we get
\begin{equation}\label{4.25}
u\in L^{\infty}(0,T_*; H_0^l(0,L)\cap H^{2l+1}(0,L)).
\end{equation}
On the other hand, $u(t,\cdot)\in H^{2l+1}(0,L)$ implies $Du(t,\cdot)\in H^{2l}(0,L)\hookrightarrow H^l(0,L)$, for all $l\in \mathbb{N}$. Then, there exists a constant $K_*$ depending on $l$, $L$, such that: (See \cite{Adams}, Theorem 4.39.)
\begin{align*}
\int_{0}^{T_*}\|u^kDu\|_{H^l}^2(t)dt&\leq K_*\int_{0}^{T_*}\|u\|_{H^l}^{2k}(t)\|Du\|_{H^l}^2(t)dt\\
&\leq K_*T_*\|u\|_{L^{\infty}(0,T_*; H^{2l+1}(0,L))}^{2(k+1)}<+\infty,
\end{align*}
therefore
\begin{equation}\label{4.26}
u^kDu\in L^2(0,T_*; H^l(0,L)) \,\,\, \mbox{for all} \,\,\, k\in \mathbb{N}.
\end{equation}
	Returning to \eqref{3.1} and taking into account \eqref{4.23},\eqref{4.26}, we find
\begin{equation}\label{4.27}
\sum_{j=1}^{l}(-1)^{j+1}D^{2j+1}u=-u_t-u^kDu\in L^2(0,T_*; H^l(0,L)).
\end{equation}
Differentiating \eqref{4.27} $l$ times with respect to $x$, we obtain 
\begin{equation*}
\sum_{l<2j\leq 2l}(-1)^{j+1}D^{(2j+1)+l}u=-\sum_{2j\leq l}(-1)^{j+1}D^{(2j+1)+l}u-D^l[u_t+u^kDu].
\end{equation*}
We estimate
\begin{align}\label{4.28}
\|D^{(2l+1)+l}u\|(t)\leq\sum_{l<2j< 2l}\|D^{(2j+1)+l}u\|(t)\notag\\+\sum_{2j\leq l}\|D^{(2j+1)+l}u\|(t)+\|D^l[u_t+u^kDu]\|(t).
\end{align}
For $l=1$, we have $\sum_{l<2j< 2l}\|D^{(2j+1)+l}u\|(t)= 0$. For $l\geq 2$, due to \eqref{2.7}, there are constants $A_1^j$, $A_2^j$ ($l<2j<2l$) depending only on $L$, $l$, such that
$$
\|D^{(2j+1)+l}u\|(t)\leq A_1^j\|D^{(2l+1)+l}u\|^{\alpha^j}(t)\|u\|^{1-\alpha^j}(t)+A_2^j\|u\|(t)\, ,$$ where $\alpha^j=\frac{(2j+1)+l}{(2l+1)+l}$. Making use of the Young inequality with $p^j=\frac{1}{\alpha^j}$, $q^j=\frac{1}{1-\alpha^j}$ and arbitrary $\epsilon>0$, we get
$$
\|D^{(2j+1)+l}u\|(t)\leq \epsilon\|D^{(2l+1)+l}u\|(t)+C_j(\epsilon)\|u\|(t)+A_2^j\|u\|(t), 
$$
where $C_j(\epsilon)=\left[q^j\left(\frac{p^j\epsilon}{({A_1^j})^{p^j}}\right)^{\frac{q^j}{p^j}}
\right]^{-1}$. Summing over $l<2j<2l$, we find
\begin{align}\label{4.29}
\sum_{l<2j<2l}\|D^{(2j+1)+l}u\|(t)\leq l\epsilon\|D^{(2l+1)+l}u\|(t)\notag\\+\sum_{l<2j<2l}(C_j(\epsilon)+A_2^j)\|u\|(t).
\end{align}
Substituting \eqref{4.29} into \eqref{4.28} and taking $\epsilon=\frac{1}{2l}$, we get
\begin{equation}\label{4.30}
\|D^{(2l+1)+l}u\|^2(t)\leq C\left(\|u\|_{H^{2l+1}}^2(t)+\|u_t+u^kDu\|_{H^l}^2(t)\right),
\end{equation}
where $C$ is a constant depending only on $L$, $l$. By \eqref{4.23},\eqref{4.25},\eqref{4.26}, it follows that $D^{(2l+1)+l}u\in L^2(0,T_*;L^2(0,L))$. Again by \eqref{2.7}, we obtain for the intermediate derivatives $D^iu$, $i=(2l+1)+1, \ldots, (2l+1)+l-1$: 
\begin{equation}\label{4.31}
\|D^iu\|^2(t)\leq C_i\left(\|D^{(2l+1)+l}u\|^2(t)+\|u\|^2(t)\right),
\end{equation}
with the constants $C_i$ depending only on $L$, $l$.  Substituting \eqref{4.30} into \eqref{4.31} and taking into account \eqref{4.23},\eqref{4.25},\eqref{4.26}, we conclude 
\begin{equation}\label{4.32}
u\in L^2(0,T_*; H^{(2l+1)+l}(0,L)).
\end{equation}
Combining \eqref{4.23},\eqref{4.25},\eqref{4.32}, we complete the proof of Theorem 4.3.
\end{proof}

{\bf Conclusions.} Making use of the  Contraction principle, we obtain  local existence and uniqueness of a regular solution for all $k\in \mathbb{N}$. We must mention a smoothing effect, first proved in \cite{kato} for the KdV equation. Roughly speaking, we proved that if $u_0\in H^{2l+1}(0,L)$, then $u(t,\cdot)\in H^{(2l+1)+l}(0,L)$, $t>0$.

\end{document}